\documentclass{amsart}
\usepackage{amssymb}
\usepackage{hyperref}
\DeclareMathOperator{\supp}{supp}
\DeclareMathOperator{\spn}{span}
\newtheorem{theorem}{Theorem}[section]
\newtheorem{lemma}[theorem]{Lemma}
\newtheorem{definition}[theorem]{Definition}
\theoremstyle{remark}\newtheorem{remark}[theorem]{Remark}
\numberwithin{equation}{section}

\title[Vertical projections in the Heisenberg group]{Vertical projections in the Heisenberg group for sets of dimension greater than 3}
\author{Terence L.~J.~Harris}
\address{Department of Mathematics, Cornell University, Ithaca, NY 14853, USA}
\email{tlh236@cornell.edu}

\begin{document} 
\begin{abstract} It is shown that vertical projections in the Heisenberg group of sets of dimension strictly greater than 3 almost surely have positive area. The proof uses the point-plate incidence method introduced by Fässler and Orponen, and also uses a similar approach to a recent maximal inequality of Zahl for fractal families of tubes. It relies on the endpoint trilinear Kakeya inequality in $\mathbb{R}^3$. Some related results are given on generic intersections with horizontal lines.  \end{abstract}
\maketitle

\section{Introduction} \begin{sloppypar} Let $\mathbb{H}$ be the Heisenberg group, identified as a set with $\mathbb{C} \times \mathbb{R}$, and equipped with the product
\[ (x,y,t) \ast (u,v, \tau) = \left(x+u, y+v, t + \tau + \frac{1}{2}\left( xv - yu \right) \right). \]
For each $\theta \in [0, \pi)$, let
\[ \mathbb{V}_{\theta} = \left\{ \left(\lambda e^{i \theta}, 0\right) \in \mathbb{C} \times \mathbb{R} : \lambda \in \mathbb{R} \right\}, \]
and let $\mathbb{V}_{\theta}^{\perp}$ be the Euclidean orthogonal complement of $\mathbb{V}_{\theta}$. Each $(z, t) \in \mathbb{H}$ can be uniquely decomposed as a product
\[ (z,t)  = P_{\mathbb{V}_{\theta}^{\perp}}(z,t) \ast P_{\mathbb{V}_{\theta}}(z,t), \]
of an element of $\mathbb{V}_{\theta}^{\perp}$ on the left with an element of $\mathbb{V}_{\theta}$ on the right. This defines the vertical projection maps $P_{\mathbb{V}_{\theta}^{\perp}}$. Let $d_{\mathbb{H}}$ be the (left-invariant) Korányi metric on $\mathbb{H}$, given by 
\[ d_{\mathbb{H}}((z,t), (\zeta, \tau)) = \left\lVert (\zeta, \tau)^{-1} \ast (z,t) \right\rVert_{\mathbb{H}}, \]
where 
\[  \left\lVert (z,t) \right\rVert_{\mathbb{H}} = \left( \lvert z \rvert^4 + 16t^2 \right)^{1/4}. \]
The Korányi metric is bi-Lipschitz equivalent to the Carnot-Carathéodory metric on $\mathbb{H}$ \cite[pp.~18--19]{capogna}, and thus induces the same Hausdorff dimension. Let ``$\dim$'' refer to the Hausdorff dimension of a set in $\mathbb{H}$ with respect to the Korányi metric. This work gives a proof of the following theorem.
\begin{theorem} \label{maintheorem} Let $A$ be an analytic subset of $\mathbb{H}$. If $\dim A > 3$ then 
\[ \mathcal{H}^3_{\mathbb{H}}\left(  P_{\mathbb{V}_{\theta}^{\perp}}(A) \right) > 0,  \]
for a.e.~$\theta \in [0, \pi)$.  \end{theorem}
In the statement of Theorem~\ref{maintheorem}, the restriction of the measure $\mathcal{H}^3_{\mathbb{H}}$ to any $\mathbb{V}_{\theta}^{\perp}$ coincides with the 2-dimensional Lebesgue measure on $\mathbb{V}_{\theta}^{\perp}$. Theorem~\ref{maintheorem} proves the $\dim A > 3$ part of a conjecture of Balogh,~Durand-Cartagena,~Fässler,~Mattila and Tyson \cite[Conjecture~1.5]{BDFMT}, who also conjectured that $\dim \left(  P_{\mathbb{V}_{\theta}^{\perp}}(A) \right) \geq \dim A$ for a.e.~$\theta$ when $\dim A \leq 3$, and proved this conjecture for $\dim A \leq 1$. The conjecture was recently proved in the range $\dim A \in [0,2] \cup \{3\}$ by Fässler and Orponen~\cite{fasslerorponen}, who introduced the method of point-plate incidences, and proved the $\dim A =3$ case by using a square function estimate for the cone of Guth, Wang and Zhang \cite{guthwangzhang} to control the average $L^2$ norms of pushforwards of discretised 3-dimensional measures. The point-line duality principle they used is due to Liu~\cite{liu}. \end{sloppypar}

The proof of Theorem~\ref{maintheorem} here also uses the point-plate incidence approach of Fässler and Orponen, but rather than using the square function estimate for the cone, it uses a broad-narrow approach to Kakeya-type inequalities for tubes arranged in fractal families of planks. This is based on recent work of Zahl \cite{zahl}, which used a broad-narrow approach to Kakeya-type inequalities for fractal families of tubes. The broad part is bounded by using the endpoint trilinear Kakeya inequality in $\mathbb{R}^3$. The non-endpoint case of the multilinear Kakeya inequality was first proved by Bennett, Carbery, and Tao~\cite{BCT}, and the endpoint case was proved by Guth~\cite{guth}. Carbery and Valdimarsson~\cite{carberyvaldimarsson} later gave a proof of the endpoint case using the Borsuk-Ulam theorem (avoiding more advanced algebraic topology), and it is their version of the inequality that will be used here.

In Theorem~\ref{exceptionalset}, it is shown that if $s>3$, and if $A \subseteq \mathbb{H}$ is $\mathcal{H}^s_{\mathbb{H}}$-measurable with $0 < \mathcal{H}^s_{\mathbb{H}}(A)  < \infty$, then there is a set $E \subseteq [0, \pi)$ of measure zero, such that any $\mathcal{H}^s_{\mathbb{H}}$-measurable subset $B \subseteq A$ with $\mathcal{H}^s_{\mathbb{H}}(B)>0$ satisfies $\mathcal{H}_{\mathbb{H}}^3\left(P_{\mathbb{V}_{\theta}^{\perp}}(B) \right)>0$ for all $\theta \in [0, \pi) \setminus E$. This generalises Theorem~\ref{maintheorem} since, by a theorem of Howroyd~\cite{howroyd}, any analytic set of infinite $\mathcal{H}^s_{\mathbb{H}}$ measure contains a subset of nonzero finite $\mathcal{H}^s_{\mathbb{H}}$ measure. This generalisation is analogous to a version of Marstrand's original projection theorem \cite[Lemma~13]{marstrand}. 

In Section~\ref{intersections}, it is shown that if $s>3$ and $A \subseteq \mathbb{H}$ is $\mathcal{H}^s_{\mathbb{H}}$-measurable with $0 < \mathcal{H}^s_{\mathbb{H}}(A)  < \infty$, then for $\left(\mathcal{H}^s_{\mathbb{H}} \times \mathcal{H}^1_E\right)$-a.e.~$(x, \theta) \in A \times [0, \pi)$, 
\[ \dim\left( A \cap P_{\mathbb{V}_{\theta}^{\perp}}^{-1}\left(P_{\mathbb{V}_{\theta}^{\perp}}(x) \right)\right) = s-3, \]
and for a.e.~$\theta \in [0,\pi)$,
\[ \mathcal{H}^3_{\mathbb{H}}\left\{ w \in \mathbb{V}_{\theta}^{\perp} : \dim\left( A \cap P_{\mathbb{V}_{\theta}^{\perp}}^{-1}(w) \right) = s-3 \right\} >0. \]
The difficulty of the intersection problem with horizontal lines is discussed briefly in~\cite{BFMT} (see the end of the introduction). The proof of the intersection theorem here is inspired by a recent general intersection theorem of Mattila~\cite[Theorem~3.1]{mattila2} (see also~\cite{mattila20}). That theorem is Euclidean, assumes an $L^2$ bound on projections of $\mathcal{H}^s\restriction_A$, and assumes a lower density assumption on $\mathcal{H}^s\restriction_A$, but in this particular instance these assumptions can be weakened (partly due the factor $\mu(\mathbb{H})c_t(\mu)^{1/2}$ appearing in \eqref{quantbound}). More importantly, the method in Section~\ref{intersections} suggests an approach to removing the lower density assumptions in \cite[Theorem~3.1]{mattila2} in a more general setting, and also suggests that the hypotheses of \cite[Theorem~3.1]{mattila2} can be generalised to $L^q$ bounds of projections for any $q>1$.

\section*{Acknowledgements} I thank Shaoming Guo for some discussions in the earlier stages of working on this problem, when I visited UW-Madison in October 2022, and I thank UW-Madison for their hospitality. I also thank the anonymous referee for pointing out some missed references and for some suggestions which improved the exposition.

\section{Preliminaries} 

Most of the background material in this section is from \cite{fasslerorponen}, but is included for completeness. Given $(z,t) \in \mathbb{H}$, let $B_E((z,t), r)$ and $B_{\mathbb{H}}((z,t),r)$ denote the Euclidean and Korányi balls around $(z,t)$ of radius $r$, respectively. For each $s \geq 0$, let $\mathcal{H}^s_E$ be the $s$-dimensional Hausdorff measure with respect to the Euclidean metric, and let $\mathcal{H}^s_{\mathbb{H}}$ be the $s$-dimensional Hausdorff measure on $\mathbb{H}$ with respect to the Korányi metric. The measures $\mathcal{H}^3_{\mathbb{H}}$ and $\mathcal{H}^2_E$ are equivalent up to a constant when restricted to any vertical plane $\mathbb{V}_{\theta}^{\perp}$. The measures $\mathcal{H}^4_{\mathbb{H}}$ and $\mathcal{H}^3_E$ are equivalent up to a constant on all of $\mathbb{H}$; by uniqueness of the Haar measure. A line $\ell$ in $\mathbb{H}$ is called horizontal if it is a left translate of a horizontal subgroup $\mathbb{V}_{\theta}$ for some $\theta \in [0,\pi)$; meaning that there exists $p \in \mathbb{H}$ such that $\ell = p \ast \mathbb{V}_{\theta}$. Given a non-negative Borel function $f$ on $\mathbb{H}$ and a horizontal line $\ell$, define 
\[ Xf(\ell) = \int_{\ell} f \, d\mathcal{H}^1_{\mathbb{H}}. \]
If $\ell = P_{\mathbb{V}_{\theta}^{\perp}}(z,t) \ast \mathbb{V}_{\theta}$ for some $(z,t)$ in the unit ball of $\mathbb{H}$, it is easy to see that $\mathcal{H}^1_E$ and $\mathcal{H}^1_{\mathbb{H}}$ are equivalent on $\ell$ up to a factor $\sim 1$. Let $\mathfrak{h}$ be the left-invariant measure on the set of horizontal lines given by 
\[ \mathfrak{h}(F) = \int_0^{\pi} \mathcal{H}^3_{\mathbb{H}}\left\{ w  \in \mathbb{V}_{\theta}^{\perp} : w \ast \mathbb{V}_{\theta} \in F \right\} \, d\theta, \]
for a Borel set $F$ of horizontal lines. Uniqueness of the measure will not be used here, but up to a constant this measure is the unique left-invariant locally finite nonzero measure on the set of horizontal lines~\cite{cheeger}; see~\cite[Lemma~2.11]{fasslerorponenrigot} for a proof of left-invariance. Given a Borel measure $\mu$ on $\mathbb{H}$ and $\theta \in [0, \pi)$, let $P_{\mathbb{V}_{\theta}^{\perp}\#} \mu$ be the pushforward of $\mu$ under $(z,t) \mapsto P_{\mathbb{V}_{\theta}^{\perp}}(z,t)$, given by $\left(P_{\mathbb{V}_{\theta}^{\perp}\#} \mu \right)(E) = \mu\left( P_{\mathbb{V}_{\theta}^{\perp}}^{-1}(E) \right)$ for any Borel set $E \subseteq \mathbb{H}$. For any two measures $\mu$ and $\nu$ on the same measurable space, the notation $\mu \ll \nu$ indicates that $\mu$ is absolutely continuous with respect to $\nu$, meaning that $\mu(A) = 0$ whenever $A$ is measurable with $\nu(A)=0$. Unless otherwise indicated, absolute continuity will be with respect to the Borel $\sigma$-algebra.

\begin{definition} Define $\ell^*: \mathbb{H} \to \mathcal{P}(\mathbb{R}^3)$ by 
\[ \ell^*(x,y,t) = \left(0, x , t-\frac{xy}{2}\right) + L_y, \]
where $L_y$ is the ``light ray'' in the light cone 
\[ \widetilde{\Gamma} := \left\{ \eta \in \mathbb{R}^3 : \eta_2^2 = 2\eta_1\eta_3 \right\}, \]
given by
\[ L_y = \left\{ \lambda \left(1, -y, \frac{y^2}{2} \right) : \lambda \in \mathbb{R} \right\}. \]
For a set $B \subseteq \mathbb{H}$, define $\ell^*(B) = \bigcup_{(z,t) \in B} \ell^*(z,t)$. Define $\ell: \mathbb{R}^3 \to \mathcal{P}(\mathbb{H})$ by 
\[ \ell(a,b,c) = \left\{ \left(as+b, s, c + \frac{bs}{2} \right) : s \in \mathbb{R} \right\}, \]
which is a horizontal line for any $(a,b,c) \in \mathbb{R}^3$ (see \eqref{usedabove1}).  \end{definition} 

The cone $\widetilde{\Gamma}$ is the image of the light cone 
\[ \Gamma := \left\{ \xi \in \mathbb{R}^3 : \xi_3^2 = \xi_1^2 + \xi_2^2 \right\}, \]
 in $\mathbb{R}^3$, under the orthogonal transformation
\[ \eta_1 = \frac{\xi_1 + \xi_3}{\sqrt{2}}, \qquad \eta_3 = \frac{-\xi_1+\xi_3}{\sqrt{2}}, \qquad \eta_2 = \xi_2. \]
This transformation is a clockwise rotation by $\pi/4$ in the $(\xi_1, \xi_3)$ plane, since 
\[ \begin{pmatrix} 1/\sqrt{2} &  1/\sqrt{2} \\ - 1/\sqrt{2} &  1/\sqrt{2}\end{pmatrix} = \begin{pmatrix} \cos(-\pi/4) & -\sin(-\pi/4) \\ \sin(-\pi/4) & \cos(-\pi/4) \end{pmatrix}, \]
where $\begin{pmatrix} \cos\theta & -\sin \theta \\ \sin \theta & \cos \theta \end{pmatrix}$ is an anticlockwise rotation in the plane by $\theta$. 

In the following discussion, the term ``light ray'' refers specifically to light rays in translates of $\widetilde{\Gamma}$. 
As $y$ varies over $\mathbb{R}$, the family of light rays $L_y$ foliate the light cone $\widetilde{\Gamma}$, except for the vertical axis $\{ (0,0,\eta_3) \in \mathbb{R}^3 : \eta_3 \in \mathbb{R}\}$. To understand the family of lines $\ell^*(\mathbb{H})$, each $(x,y,t) \in \mathbb{H}$ can be uniquely written as $(x,y,t) = (u, 0, v) \ast (0, y,0)$ with $u,v \in \mathbb{R}$, and for each $u,v,y \in \mathbb{R}$ there is a unique $x,t \in \mathbb{R}$ such that $(x,y,t) = (u,0,v) \ast (0,y,0)$. Under this correspondence, 
\[ \ell^*(x,y,t) = (0, u, v) + L_y. \]
Therefore, if we fix $u,v \in \mathbb{R}$ and vary $y$ over $\mathbb{R}$, we get the family of light rays in $\mathbb{R}^3$ passing through $(0,u,v)$, foliating the translated cone $(0,u,v) + \widetilde{\Gamma}$, minus the vertical light ray through $(0,u,v)$. Since any non-vertical light ray (i.e.~translate of $L_y$) must intersect the plane $\{ \eta \in \mathbb{R}^3  : \eta_1 = 0 \}$ at some point $(0,u,v)$, by varying $u$ and $v$ we get that every non-vertical light ray is of the form $\ell^*(x,y,t)$ for some $(x,y,t) \in \mathbb{H}$. Thus the map $\ell^*$ is a one-to-one correspondence between points in $\mathbb{H}$ and non-vertical light rays in $\mathbb{R}^3$. 

In a similar way, by observing that for any $(a,b,c) \in \mathbb{R}^3$,
\begin{equation} \label{usedabove1} \ell(a,b,c) = (b,0,c) \ast \spn(a,1,0), \end{equation}
the map $\ell$ is a one-to-one correspondence between points in $\mathbb{R}^3$ and horizontal lines in $\mathbb{H}$ minus the left translates of $\mathbb{V}_0 =\spn(1,0,0)$. 

The following lemma is the point-line duality principle from \cite{fasslerorponen}. The proof follows straightforwardly from the definitions. 
\begin{lemma}[{\cite[Lemma~4.11]{fasslerorponen}}] \label{pointline} Let $p \in \mathbb{R}^3$ and $p^* \in \mathbb{H}$. Then 
\[ p \in \ell^*(p^*) \quad \text{ if and only if } \quad p^* \in \ell(p). \] \end{lemma}

The following lemma was shown in \cite[Section~4]{fasslerorponen} in the case $q=2$. 
\begin{lemma} \label{Xray} Let $f$ be a non-negative Borel function on $\mathbb{H}$ such that $\int_{\mathbb{H}} f(x) \, dx < \infty$, and let $\mu_f$ be the Borel measure such that the Radon-Nikodym derivative of $\mu_f$ with respect to the Lebesgue measure on $\mathbb{H}$ is equal to $f$. Then:
\begin{enumerate}
\item  For any $q \in [1, \infty)$,
\[ \int_0^{\pi} \left\lVert P_{\mathbb{V}_{\theta}^{\perp} \#} \mu_f \right\rVert_{L^q\left(\mathcal{H}^3_{\mathbb{H}}\right) }^q \, d\theta = \int \left\lvert Xf(\ell) \right\rvert^q \, d\mathfrak{h}(\ell). \] 
\item For any $q \in [1, \infty)$ and $\varepsilon>0$, 
\[ \int_{\varepsilon}^{\pi-\varepsilon} \left\lVert P_{\mathbb{V}_{\theta}^{\perp} \#} \mu_f \right\rVert_{L^q\left(\mathcal{H}^3_{\mathbb{H}}\right) }^q \, d\theta \sim_{\varepsilon} \int_{\mathcal{L}_{\varepsilon}} \left\lvert Xf(\ell(p)) \right\rvert^q \, d\mathcal{H}_E^3(p), \]
where $\mathcal{L}_{\varepsilon}$ is the set of $p \in \mathbb{R}^3$ such that $\ell(p) = (z,t) \ast \mathbb{V}_{\theta}$ for some $(z,t) \in \mathbb{H}$ and $\theta \in [\varepsilon, \pi-\varepsilon]$. \end{enumerate}
 \end{lemma}
\begin{proof} By the Euclidean coarea formula (see e.g.~\cite[Theorem~3.11]{evans}), the Radon-Nikodym derivative of $P_{\mathbb{V}_{\theta}^{\perp} \#} \mu_f$ with respect to the measure $\mathcal{H}^3_{\mathbb{H}}$ on $\mathbb{V}_{\theta}^{\perp}$ is given by 
\[ \left( P_{\mathbb{V}_{\theta}^{\perp} \#} \mu_f \right)(w) = \int_{P_{\mathbb{V}_{\theta}^{\perp}}^{-1}(w) } f \, d\mathcal{H}^1_{\mathbb{H}}; \]
since the Jacobian factor from the Euclidean coarea formula cancels with the factor obtained by changing $\mathcal{H}^1_E$ to $\mathcal{H}^1_{\mathbb{H}}$. By the formula for $L^q$ norms in terms of the distribution function, and by Fubini's theorem,
\begin{align*} &\int_0^{\pi} \left\lVert P_{\mathbb{V}_{\theta}^{\perp} \#} \mu_f \right\rVert_{L^q\left(\mathcal{H}^3_{\mathbb{H}}\right) }^q \, d\theta \\
&\quad= q \int_0^{\infty}  \lambda^{q-1}  \int_0^{\pi}\mathcal{H}^3_{\mathbb{H}}\left\{ w\in \mathbb{V}_{\theta} : \int_{P_{\mathbb{V}_{\theta}^{\perp}}^{-1}(w)} f \, d\mathcal{H}^1_{\mathbb{H}} > \lambda \right\} \, d\theta \, d\lambda  \\
&\quad= q \int_0^{\infty}  \lambda^{q-1} \mathfrak{h}\left\{ \ell : Xf(\ell) > \lambda \right\} \, d\lambda \\
&\quad = \int \left\lvert Xf(\ell) \right\rvert^q \, d\mathfrak{h}(\ell).  \end{align*} 

Similarly to \cite[Eq.~4.15]{fasslerorponen}, the second part follows from the first part; using that on horizontal lines $(z,t) \ast \mathbb{V}_{\theta}$ with $\theta \in [\varepsilon, \pi-\varepsilon]$, $\mathfrak{h}$ is equivalent (up to constant depending on $\varepsilon$) to the pushforward of Lebesgue measure under the map $p \mapsto \ell(p)$ (this is a straightforward argument that follows from the definition of the function $\ell$).  \end{proof}

\begin{lemma} \label{translation} Let $f$ be a non-negative Borel function on $\mathbb{H}$ such that $\int_{\mathbb{H}} f(x) \, dx < \infty$, and let $\mu_f$ be the Borel measure such that the Radon-Nikodym derivative of $\mu_f$ with respect to the Lebesgue measure on $\mathbb{H}$ is equal to $f$. Then for any $q \in [1, \infty)$ and any $p \in \mathbb{H}$, 
\[ \int_0^{\pi} \left\lVert P_{\mathbb{V}_{\theta}^{\perp} \#}  L_{p\#}\mu_f \right\rVert_{L^q\left(\mathcal{H}^3_{\mathbb{H}}\right) }^q \, d\theta = \int_0^{\pi} \left\lVert P_{\mathbb{V}_{\theta}^{\perp} \#} \mu_f \right\rVert_{L^q\left(\mathcal{H}^3_{\mathbb{H}}\right) }^q \, d\theta, \]
where $L_p(z,t) = p \ast (z,t)$. \end{lemma}
\begin{proof} The Radon-Nikodym derivative of $L_{p\#}\mu_f$ is $f \circ L_p^{-1}$, since left translation has Jacobian equal to 1. By Lemma~\ref{Xray},
\[ \int_0^{\pi}   \left\lVert P_{\mathbb{V}_{\theta}^{\perp} \#}  L_{p\#}\mu_f\right\rVert_{L^q\left(\mathcal{H}^3_{\mathbb{H}}\right) }^q \, d\theta =  \int \left\lvert \int_{\ell} (f \circ L_p^{-1}) \, d\mathcal{H}^1_{\mathbb{H}} \right\rvert^q \, d\mathfrak{h}(\ell). \]
Since $\mathcal{H}^1_{\mathbb{H}}$ is left-invariant, and $\mathfrak{h}$ is left-invariant, the right-hand side satisfies
\begin{align*}  \int \left\lvert \int_{\ell} (f \circ L_p^{-1}) \, d\mathcal{H}^1_{\mathbb{H}} \right\rvert^q \, d\mathfrak{h}(\ell) &= \int \left\lvert \int_{\ell} f \, \, d\mathcal{H}^1_{\mathbb{H}} \right\rvert^q \, d\mathfrak{h}(\ell) \\
&=  \int_0^{\pi}   \left\lVert P_{\mathbb{V}_{\theta}^{\perp} \#} \mu_f \right\rVert_{L^q\left(\mathcal{H}^3_{\mathbb{H}}\right) }^q \, d\theta. \end{align*}
This proves the lemma.
 \end{proof}

The next two lemmas are very similar to the results in Section~4 of \cite{fasslerorponen}, which also includes figures and a heuristic discussion. The proofs here are straightforward computations, but see \cite{fasslerorponen} for more conceptual proofs of similar results. 

\begin{sloppypar} \begin{lemma} \label{broadplank} There exists an absolute constant $C>0$ such that the following holds. If $(x,y,t) \in B_{\mathbb{H}}(0,1)$ and $\delta \in (0,1)$, then $\ell^*(B_{\mathbb{H}}(x,y,t), \delta) \cap B_E(0,2)$ is contained in a box of dimensions 
\[ C \times C \delta \times C \delta^2, \]
with long (radial) direction parallel to $\left(1,-y,y^2/2\right)$, medium (tangential) direction parallel to $\left(y, 1- y^2/2, -y\right)$, and short (normal) direction parallel to $\left(y^2/2, y, 1\right)$. \end{lemma} \end{sloppypar}
\begin{proof} Any point in $B_{\mathbb{H}}\left((x,y,t), \delta\right)$ can be written as
\[ (x,y,t) \ast \left(x',y',t'\right) \in B_{\mathbb{H}}\left((x,y,t), \delta\right), \]
where $\left(x',y',t'\right) \in B_{\mathbb{H}}(0,\delta)$, so let such a point be given. By definition, 
\begin{multline*} \ell^*((x,y,t) \ast \left(x',y',t'\right) ) = \ell^*\left( x+x', y+y', t+ t' + \frac{xy'-yx'}{2} \right) \\
= \left( 0,  x, t- \frac{xy}{2} \right)  + \left( 0, x', t'- \frac{x'y'}{2} - x'y \right) + \spn\left( 1, -y - y' , \frac{(y+y')^2}{2} \right). \end{multline*} 
It therefore suffices to prove that if 
\[ v = \left( 0, x', t'- \frac{x'y'}{2} - x'y \right) + \lambda\left( 1, -y - y' , \frac{(y+y')^2}{2} \right), \]
where  
\[ (x,y,t) \in B_{\mathbb{H}}(0,1), \quad \left(x',y',t'\right) \in B_{\mathbb{H}}(0, \delta), \quad \lvert\lambda\rvert  \leq 2, \]
then $v$ is contained in a box centred at the origin of the specified dimensions. Clearly $\left\lvert \left\langle v,\left(1,-y,y^2/2\right) \right\rangle \right\rvert \lesssim 1$, which is the bound along the longest direction of the box. For the medium direction, 
\[ v = \lambda\left( 1, -y, y^2/2 \right) + O(\delta), \]
where the ``$O(\delta)$'' is with respect to Euclidean distance. Thus 
\[ \left\lvert \left\langle v, \left(y, 1- y^2/2, -y\right) \right\rangle \right\rvert = O(\delta) + \lvert\lambda\rvert \left\lvert \left\langle \left( 1, -y, y^2/2 \right), \left(y, 1- y^2/2, -y\right) \right\rangle \right\rvert. \]
But $\left\langle \left( 1, -y, y^2/2 \right), \left(y, 1- y^2/2, -y\right) \right\rangle=0$, so this proves the bound for the medium direction. For the short direction, 
\[ v = \left(0, x', -x'y \right) + \lambda \left( 1, -y - y', \frac{y^2}{2} + yy' \right) + O\left(\delta^2\right), \]
and hence
\begin{multline*} \left\lvert \left\langle v, \left(y^2/2, y, 1\right) \right\rangle \right\rvert  = O(\delta^2) + \\
 \left\lvert \left\langle \left(0, x', -x'y \right) + \lambda \left( 1, -y - y', \frac{y^2}{2} + yy' \right), \left(\frac{y^2}{2}, y, 1\right) \right\rangle \right\rvert. \end{multline*}
But  
\begin{multline*} \left\langle \left(0, x', -x'y \right) + \lambda \left( 1, -y - y', \frac{y^2}{2} + yy' \right), \left(\frac{y^2}{2}, y, 1\right) \right\rangle \\
=  \lambda \left\langle   \left( 1, -y - y', \frac{y^2}{2} + yy' \right), \left(\frac{y^2}{2}, y, 1\right) \right\rangle = 0, \end{multline*}
so this proves the bound for the short direction. \end{proof}

\begin{sloppypar} \begin{lemma} \label{localisation} There exists an absolute constant $C>0$ such that the following holds. Let $K \geq 1$ and let $\rho \in (0,1)$. Let $(x_1, y_1, t_1), (x_2,y_2, t_2) \in B_{\mathbb{H}}(0,1/10)$, and suppose that $y_0 \in \mathbb{R}$ is such that $\lvert y_0-y_1 \rvert \leq \rho$ and $\lvert y_0-y_2\lvert \leq \rho$. Suppose also that $\ell^*(x_1, y_1, t_1) \cap B_E(0,2)$ and $\ell^*(x_2, y_2, t_2) \cap B_E(0,2)$ are both contained in a single box of dimensions 
\[ K \times K \rho \times K \rho^2, \]
with long direction parallel to $\left(1,-y_0,y_0^2/2\right)$, medium direction parallel to $\left(y_0, 1- y_0^2/2, -y_0\right)$, and short direction parallel to $\left(y_0^2/2, y_0, 1\right)$. Then 
\[ d_{\mathbb{H}}((x_1, y_1, t_1), (x_2, y_2, t_2)) \leq C K \rho. \]  \end{lemma}
\begin{remark} In the application of Lemma~\ref{localisation}, $K$ will be a large constant. \end{remark}
\begin{proof}[Proof of Lemma~\ref{localisation}] Let $K$ be given. The points $\left( 0, x_1, t_1 - \frac{x_1y_1}{2} \right)$ and $\left( 0, x_2, t_2 - \frac{x_2y_2}{2} \right)$ are in $\ell^*(x_1, y_1, t_1) \cap B_E(0,2)$ and $\ell^*(x_2, y_2, t_2) \cap B_E(0,2)$ respectively, by the assumption that $(x_1, y_1, t_1), (x_2,y_2, t_2) \in B_{\mathbb{H}}(0,1/10)$. Hence
\begin{equation} \label{medium} \left\lvert \left\langle\left( 0, x_1-x_2, t_1-t_2 -\frac{x_1y_1}{2} +\frac{x_2y_2}{2} \right), \left(y_0, 1- \frac{y_0^2}{2}, -y_0\right) \right\rangle\right\rvert \lesssim K \rho, \end{equation}
and
\begin{equation} \label{short} \left\lvert \left\langle\left( 0, x_1-x_2, t_1-t_2 -\frac{x_1y_1}{2} +\frac{x_2y_2}{2} \right), \left(\frac{y_0^2}{2}, y_0, 1\right) \right\rangle\right\rvert \lesssim K \rho^2. \end{equation}
Adding a $y_0$-multiple of the inner product in \eqref{short} to the inner product in \eqref{medium} yields
\begin{equation} \label{intermediate} \lvert x_1 - x_2\rvert  \lesssim K \rho. \end{equation}
Moreover, \eqref{short} can be written as 
\[ \left\lvert t_1 - t_2  + y_0(x_1-x_2)  -\frac{x_1y_1}{2} +\frac{x_2y_2}{2}\right\rvert \lesssim K \rho^2. \]
By \eqref{intermediate} and the condition $\lvert y_0 - y_1 \rvert \leq \rho$, this implies that 
\begin{equation} \label{simplerversion} \left\lvert t_1 - t_2  + y_1(x_1-x_2)  -\frac{x_1y_1}{2} +\frac{x_2y_2}{2}\right\rvert \lesssim K \rho^2. \end{equation}
But 
\begin{multline} \label{pause371} t_1 - t_2  + y_1(x_1-x_2)  -\frac{x_1y_1}{2} +\frac{x_2y_2}{2} \\
= t_1 - t_2  + \frac{x_1y_2}{2} - \frac{x_2y_1}{2} + \frac{(x_1-x_2)(y_1-y_2)}{2}. \end{multline}
Therefore, substituting \eqref{pause371} into \eqref{simplerversion}, using \eqref{intermediate} and the inequality $\lvert y_1-y_2\rvert  \leq 2 \rho$, gives 
\[ \left\lvert t_1 - t_2  + \frac{x_1y_2}{2} - \frac{x_2y_1}{2} \right\rvert \lesssim K \rho^2. \]
Hence 
\begin{multline*} d_{\mathbb{H}}\left((x_1,y_1,t_1), (x_2, y_2, t_2)\right) \\
\sim \lvert x_1-x_2\rvert  + \lvert y_1-y_2\rvert  + \left\lvert t_1 - t_2  + \frac{x_1y_2}{2} - \frac{x_2y_1}{2} \right\rvert^{1/2} \lesssim K \rho. \qedhere \end{multline*}\end{proof}\end{sloppypar} 

For reference, the version of the trilinear Kakeya inequality in $\mathbb{R}^3$ from \cite{carberyvaldimarsson} is stated below. By a ``1-tube'' is meant a 1-neighbourhood of an infinite line in $\mathbb{R}^3$. If $T$ is a 1-tube, let $v_T$ be a unit vector parallel to the direction of $T$. Given unit vectors $v_1,v_2,v_3$ in $\mathbb{R}^3$, let $\left\lvert \det(v_1, v_2, v_3 ) \right\rvert = \left\lvert v_1 \wedge v_2 \wedge v_3 \right\rvert$ be the volume of the parallelepiped with sides equal to the vectors $v_1, v_2, v_3$. 
\begin{theorem}[{\cite{carberyvaldimarsson}}] \label{kakeya} There exists a constant $C>0$ such that the following holds. If $\mathbb{T}_1, \mathbb{T}_2, \mathbb{T}_3$ are finite sets of 1-tubes in $\mathbb{R}^3$, and $\left\{ a_{T_j} \right\}_{T_j \in \mathbb{T}_j, j \in \{1,2,3\}}$ are non-negative real numbers, then 
\begin{multline*} \int_{\mathbb{R}^3} \left( \sum_{T_1 \in \mathbb{T}_1} \sum_{T_2 \in \mathbb{T}_2} \sum_{T_3 \in \mathbb{T}_3} a_{T_1} a_{T_2} a_{T_3} \left\lvert v_{T_1} \wedge v_{T_2} \wedge v_{T_3} \right\rvert \chi_{T_1} \chi_{T_2} \chi_{T_3} \right)^{1/2} \, dx \\
\leq C\left( \prod_{j=1}^3 \sum_{T_j \in \mathbb{T}_j} a_{T_j}\right)^{1/2}. \end{multline*}  \end{theorem}

\section{Projections}  Given a Borel measure $\mu$ on $\mathbb{H}$, define
\[ c_t(\mu) = \sup_{x \in \mathbb{H}, r >0} \frac{ \mu\left( B_{\mathbb{H}}(x,r) \right)}{r^t}. \]
The following inequality will be used to deduce both the projection and intersection theorems stated in the introduction. 
\begin{theorem} \label{quantitativetheorem} Let $\mu$ be a Borel measure supported in $B_{\mathbb{H}}(0,1)$. Suppose that $t > 3$ and $c_t(\mu) < \infty$. Then $P_{\mathbb{V}_{\theta}^{\perp}\#} \mu \ll\mathcal{H}^3_{\mathbb{H}}$ for a.e.~$\theta \in [0, \pi)$, and
\begin{equation} \label{quantbound} \int_0^{\pi} \left\lVert P_{\mathbb{V}_{\theta}^{\perp}\#} \mu \right\rVert_{L^{3/2}\left(\mathcal{H}^3_{\mathbb{H}}\right)}^{3/2} \, d\theta \leq C_t\mu(\mathbb{H}) c_t(\mu)^{1/2}, \end{equation}
where $C_t$ is a constant depending only on $t$. 
\begin{remark} It is an open problem whether a corresponding $L^q$ inequality holds for any $q> 3/2$, if the right-hand side is replaced by  $C_t\mu(\mathbb{H}) c_t(\mu)^{q-1}$ or any constant $C_{\mu,t}$. The results in \cite{fasslerorponen} suggest a possible conjecture\footnote{This conjecture is not my own; I believe it is originally due to the authors of \cite{fasslerorponen}.} that the measures $P_{\mathbb{V}_{\theta}^{\perp}\#} \mu$ are almost surely in $L^2$, but I do not know if the right-hand side in the corresponding inequality could be expected to be $C_t\mu(\mathbb{H}) c_t(\mu)$.  \end{remark}
  \end{theorem}
	\begin{proof}[Proof of Theorem~\ref{quantitativetheorem}] It may be assumed that $t \leq 4$. Let $q=3/2$.	By a duality argument (which will be shown at the conclusion of the proof), it will suffice to show that for any $\delta \in (0,1)$,
\begin{equation} \label{claimed7} \int_0^{\pi} \left\lVert P_{\mathbb{V}_{\theta}^{\perp}\#} (\mu \ast_{\mathbb{H}} \eta_{\delta}) \right\rVert_{L^q\left(\mathcal{H}^3_{\mathbb{H}}\right)}^q \, d\theta \lesssim \mu(\mathbb{H})c_t(\mu)^{q-1}, \end{equation}
where $\eta$ is a fixed smooth non-negative bump function supported in $B_{\mathbb{H}}(0,1)$, with $\eta \sim 1$ on $B_{\mathbb{H}}(0,1/2)$, $\int_{\mathbb{H}} \eta = 1$, and where
\[ \eta_{\delta}(z,t) := \frac{1}{\delta^4} \eta\left(\frac{z}{\delta}, \frac{t}{\delta^2}\right). \]
The Heisenberg convolution is given by
\[ \left(\mu \ast_{\mathbb{H}} f\right)(z,t) = \int f\left((\zeta, \tau)^{-1} \ast (z,t)\right) \, d\mu(\zeta, \tau). \]
For any $\delta>0$, the measure $\mu \ast_{\mathbb{H}} \eta_{\delta}$ satisfies $\left(\mu \ast_{\mathbb{H}} \eta_{\delta}\right)(\mathbb{H}) = \mu(\mathbb{H})$, and 
\begin{equation} \label{alphaconv} c_t\left( \mu \ast_{\mathbb{H}} \eta_{\delta} \right) \lesssim c_t(\mu). \end{equation}
To see that \eqref{alphaconv} holds, let $r>0$ and let $(z_0, t_0) \in \mathbb{H}$. If $r>\delta$, then 
\begin{align*}  \left(\mu \ast_{\mathbb{H}} \eta_{\delta}\right)\left( B_{\mathbb{H}}((z_0,t_0), r) \right) &= \int_{B_{\mathbb{H}}((z_0,t_0), r)} \int \eta_{\delta}\left( (\zeta,\tau)^{-1} \ast (z,t) \right) \, d\mu(\zeta, \tau) \, dz \, dt \\
&\leq \int_{B_{\mathbb{H}}((z_0,t_0), 2r)} \int \eta_{\delta}\left( (\zeta,\tau)^{-1} \ast (z,t) \right) \, dz \, dt \, d\mu(\zeta, \tau) \\
&= \mu\left(B_{\mathbb{H}}((z_0,t_0), 2r)\right)\\
&\leq 2^t c_t(\mu) r^t. \end{align*}
If $r \leq \delta$, then 
\begin{align*}  \left(\mu \ast_{\mathbb{H}} \eta_{\delta}\right)\left( B_{\mathbb{H}}((z_0,t_0), r) \right) &= \int_{B_{\mathbb{H}}((z_0,t_0), r)} \int \eta_{\delta}\left( (\zeta,\tau)^{-1} \ast (z,t) \right) \, d\mu(\zeta, \tau) \, dz \, dt \\
&\lesssim \frac{1}{\delta^4}\int_{B_{\mathbb{H}}((z_0,t_0), r)} \int_{B_{\mathbb{H}}((z,t), \delta)} d\mu(\zeta, \tau)  \, dz \, dt \\
&\lesssim r^4 \delta^{t-4} c_t(\mu) \\
&\leq c_t(\mu) r^t. \end{align*}
This verifies \eqref{alphaconv}. 

If $\mathcal{B}$ is a boundedly overlapping cover of $B_{\mathbb{H}}(0,2)$ by Korányi $\delta$-balls, then 
\begin{equation} \label{smoothing} \left(\mu \ast_{\mathbb{H}} \eta_{\delta}\right)(z,t) \lesssim \frac{1}{\delta^4} \sum_{B \in \mathcal{B}} \mu(B)\chi_{2B}(z,t)  \lesssim \left(\mu \ast_{\mathbb{H}} \eta_{6\delta}\right)(z,t), \end{equation}
for all $(z,t) \in \mathbb{H}$. By \eqref{smoothing} and by dilating by a factor $\sim 1$ (see \eqref{starsss}), to prove \eqref{claimed7} it suffices to show that for any $\delta \in (0,1)$,
\begin{equation} \label{claimed2} \int_0^{\pi} \left\lVert P_{\mathbb{V}_{\theta}^{\perp}\#} \nu \right\rVert_{L^q\left(\mathcal{H}^3_{\mathbb{H}}\right)}^q \, d\theta \leq C_t \nu(\mathbb{H})c_t(\nu)^{q-1}, \end{equation}
whenever $\nu$ is a linear combination of characteristic functions over a disjoint family of Korányi $\delta$-balls in $B_{\mathbb{H}}(0,c)$, where $c$ is some small constant to be chosen. 

The inequality \eqref{claimed2} is trivial if $\delta \gtrsim 1$. Let $\delta>0$, and assume inductively that \eqref{claimed2} holds for all $\widetilde{\delta} \geq \delta/\rho$, where $\rho \in (0,1)$ is a small constant. The precise choice of $\rho$ will be made such that $\rho^{(t-3)/2} < c_0$ for some small absolute constant $c_0$, and then $C_t$ will be chosen large enough to ensure that $C_t > C_0 \rho^{-2}$ for some large absolute constant $C_0$ (the reasons for these choices will become evident throughout the proof). By scaling it may be assumed that $\nu$ is a probability measure, and by rotational symmetry it may be assumed that
\[ \int_0^{\pi} \left\lVert P_{\mathbb{V}_{\theta}^{\perp}\#} \nu \right\rVert_{L^q\left(\mathcal{H}^3_{\mathbb{H}}\right)}^q \, d\theta \lesssim \int_{\pi/4}^{3\pi/4} \left\lVert P_{\mathbb{V}_{\theta}^{\perp}\#} \nu \right\rVert_{L^q\left(\mathcal{H}^3_{\mathbb{H}}\right)}^q \, d\theta. \]
Write
\[ \nu = \frac{1}{\delta^4 \mathcal{H}_{\mathbb{H}}^4(B_{\mathbb{H}}(0,1))} \sum_{ B \in \mathcal{B}} a_B \chi_B, \]
where $\mathcal{B}$ is a disjoint collection of Korányi $\delta$-balls in $B_{\mathbb{H}}(0,c)$, and the $a_B$ are positive coefficients which sum to 1. Let $\mathcal{L}_{\angle}$ be the set of horizontal lines $\ell$ such that $\ell = (z,t) \ast \mathbb{V}_{\theta}$ for some $(z,t) \in \mathbb{H}$ and $\theta \in [\pi/4, 3\pi/4]$. By Lemma~\ref{pointline} and Lemma~\ref{Xray},
\begin{align} \notag \int_{\pi/4}^{3\pi/4} \left\lVert P_{\mathbb{V}_{\theta}^{\perp}\#} \nu \right\rVert_{L^q\left(\mathcal{H}^3_{\mathbb{H}}\right)}^q \, d\theta &\sim \int_{\mathcal{L}_{\angle}} \left\lvert X \nu(\ell) \right\rvert^q \, d\mathfrak{m}(\ell)\\
\notag  &\sim \frac{1}{ \delta^{4q}} \int_{\mathcal{L}_{\angle}} \left\lvert \sum_{B \in \mathcal{B}} a_B\mathcal{H}_{\mathbb{H}}^1(\ell \cap B) \right\rvert^q  \, d\mathfrak{m}(\ell)\\
\label{reversible2} &\lesssim \frac{1}{ \delta^{3q}} \int_{\{p \in B_E(0,2) : \ell(p) \in \mathcal{L}_{\angle}\}} \left( \sum_{B \in \mathcal{B}: B \cap \ell(p) \neq \emptyset } a_B\right)^q  \, d\mathcal{H}_E^3(p)\\
\label{quadriplestar} &=  \frac{1}{ \delta^{3q} }\left\lVert \sum_{B \in \mathcal{B}} a_B\chi_{\ell^*(B) } \right\rVert_{L^q\left(B_E(0,2) \cap \ell^{-1}(\mathcal{L}_{\angle})\right)}^q, \end{align}
where $\mathfrak{m}$ is the pushforward of the Lebesgue measure on $\mathbb{R}^3$ under $p \mapsto \ell(p)$. The above used the observation (from~\cite{fasslerorponen}) that if $\ell(p) \cap B_{\mathbb{H}}(0,c) \neq \emptyset$ and $\lvert p_1 \rvert \leq 1$, then $p \in B_E(0,2)$ provided $c$ is chosen small enough (here $p=(p_1,p_2,p_3)$). 

 \begin{sloppypar} Let $\{\tau\}$ be the standard boundedly overlapping covering of $\widetilde{\Gamma} \cap B_E(0,2) \setminus B_E(0,1/2)$ by (tangential) boxes of dimensions $\sim \rho \times \rho^2 \times 1$, recalling that
\[ \widetilde{\Gamma} = \left\{ (\eta_1,\eta_2, \eta_3) \in \mathbb{R}^3 : \eta_2^2 = 2\eta_1\eta_3 \right\}, \]
is a clockwise rotation of the standard light cone
\[ \Gamma =  \left\{ (\xi_1, \xi_2, \xi_3) \in \mathbb{R}^3 : \xi_3^2 = \xi_1^2+\xi_2^2 \right\}, \]
by $\pi/4$ in the $(\xi_1, \xi_3)$-plane. To obtain this covering, cover the unit circle $S^1 \subseteq \mathbb{R}^2$ by rectangles of dimensions $\sim \rho \times \rho^2$ tangent to $S^1$, extend these radially to get a cover of $\Gamma \cap B_E(0,2) \setminus B_E(0,1/2)$, and rotate to obtain a covering of $\widetilde{\Gamma} \cap B_E(0,2) \setminus B_E(0,1/2)$. \end{sloppypar}

Call $p \in B_E(0,2)$ ``narrow'' if there is a 2-dimensional subspace $V$ of $\mathbb{R}^3$ (depending on $p$), such that 
\[ \sum_{B \in \mathcal{B}} a_B \chi_{\ell^*(B)}(p) \leq 2 \sum_{\substack{B \in \mathcal{B} \\ u(\ell^*(B)) \in \mathcal{N}_{\rho^2}(V)}} a_B \chi_{\ell^*(B)}(p), \]
where 
\[ u(\ell^*(B)) := \frac{\left(1, -y_B, y_B^2/2\right)}{\left\lvert\left(1, -y_B, y_B^2/2\right)\right\rvert} \]
is a unit vector parallel to the direction of $\ell^*\left(x_B, y_B, t_B\right)$, and where $\left(x_B,y_B,t_B\right)$ is the centre of $B$. For any 2-dimensional subspace $V$ of $\mathbb{R}^3$, there are $\lesssim 1$ boxes $\tau$ intersecting the $\rho^2$-neighbourhood of $V$; due to the curvature of $\Gamma$. Hence, if $p$ is narrow, then
\begin{equation} \label{startwo} \left\lvert \sum_{B \in \mathcal{B}} a_B \chi_{\ell^*(B)}(p) \right\rvert^q \lesssim \sum_{\tau} \left(\sum_{B \in \mathcal{B} : u(\ell^*(B)) \in \tau} a_B\chi_{\ell^*(B)}(p)\right)^q. \end{equation}
If $p$ is not narrow then it is called ``broad''. If $p$ is broad, then
\begin{multline} \label{triplestar} \sum_{B \in \mathcal{B}} a_B\chi_{\ell^*(B)}(p) \lesssim \rho^{-4/3} \times \Bigg( \sum_{B_1 \in \mathcal{B}} \sum_{B_2 \in \mathcal{B}} \sum_{B_3 \in \mathcal{B}} a_{B_1}a_{B_2}a_{B_3}\\
\chi_{\ell^*(B_1) }(p)\chi_{\ell^*(B_2) }(p)\chi_{\ell^*(B_3)}(p) \left\lvert u(\ell^*(B_1)) \wedge u(\ell^*(B_2)) \wedge  u(\ell^*(B_3)) \right\rvert \Bigg)^{1/3}. \end{multline}
This can be shown as follows. Write 
\begin{multline} \label{pause1171}  \sum_{B \in \mathcal{B}} a_B\chi_{\ell^*(B)}(p) = \\
\Bigg( \sum_{B_1 \in \mathcal{B}} \sum_{B_2 \in \mathcal{B}} \sum_{B_3 \in \mathcal{B}} a_{B_1}a_{B_2}a_{B_3} \chi_{\ell^*(B_1) }(p)\chi_{\ell^*(B_2) }(p)\chi_{\ell^*(B_3)}(p)   \Bigg)^{1/3}. \end{multline}
Since $p$ is broad, for each $B_1,B_2 \in \mathcal{B}$, the main contribution in the sum over $B_3$ comes from those $B_3$ with $u(\ell^*(B_3))$ not contained in the $\rho^2$-neighbourhood of the span of $u(\ell^*(B_2))$, and these $B_3$ satisfy $\left\lvert u(\ell^*(B_2)) \wedge u(\ell^*(B_3)) \right\rvert \geq \widetilde{c} \rho^2$ for some absolute constant $\widetilde{c}$. Thus 
\begin{multline} \label{pause1172} \eqref{pause1171} \lesssim \Bigg( \sum_{B_1 \in \mathcal{B}} \sum_{B_2 \in \mathcal{B}} \sum_{B_3 \in \mathcal{B} : \left\lvert u(\ell^*(B_2)) \wedge u(\ell^*(B_3)) \right\rvert \geq \widetilde{c} \rho^2}\\
 a_{B_1}a_{B_2}a_{B_3} \chi_{\ell^*(B_1) }(p)\chi_{\ell^*(B_2) }(p)\chi_{\ell^*(B_3)}(p)   \Bigg)^{1/3}, \end{multline}
The sum over $B_1$ can be interchanged with the other two sums. Since $p$ is broad, for each $B_2,B_3 \in \mathcal{B}$ with $\left\lvert u(\ell^*(B_2)) \wedge u(\ell^*(B_3)) \right\rvert \geq \widetilde{c} \rho^2$, the main contribution in the sum over $B_1$ comes from those $B_1$ with $u(\ell^*(B_1))$ not contained in the $\rho^2$-neighbourhood of the plane spanned by $u(\ell^*(B_2))$ and $u(\ell^*(B_3))$, and these $B_1$ satisfy $\left\lvert u(\ell^*(B_1)) \wedge u(\ell^*(B_2)) \wedge u(\ell^*(B_3)) \right\rvert \gtrsim \rho^4$. Thus 
\begin{multline*} \eqref{pause1172} \lesssim \rho^{-4/3} \times \Bigg( \sum_{B_1 \in \mathcal{B}} \sum_{B_2 \in \mathcal{B}} \sum_{B_3 \in \mathcal{B}} a_{B_1}a_{B_2}a_{B_3}\\
\chi_{\ell^*(B_1) }(p)\chi_{\ell^*(B_2) }(p)\chi_{\ell^*(B_3)}(p) \left\lvert u(\ell^*(B_1)) \wedge u(\ell^*(B_2)) \wedge  u(\ell^*(B_3)) \right\rvert \Bigg)^{1/3}. \end{multline*}
This verifies \eqref{triplestar}. 

Since the broad and narrow points partition $B_E(0,2)$,
\begin{multline} \label{broadnarrow2} \left\lVert \sum_{B \in \mathcal{B}} a_B\chi_{\ell^*(B) } \right\rVert_{L^q(B_E(0,2) \cap \ell^{-1}(\mathcal{L}_{\angle}))} \\ \leq
 \left\lVert \chi_{\text{broad}} \sum_{B \in \mathcal{B}} a_B\chi_{\ell^*(B) } \right\rVert_{L^q(B_E(0,2))} \\
+ \left\lVert \chi_{\text{narrow}} \sum_{B \in \mathcal{B}} a_B\chi_{\ell^*(B) } \right\rVert_{L^q(B_E(0,2) \cap \ell^{-1}(\mathcal{L}_{\angle}))}. \end{multline}

\begin{sloppypar} Suppose first that the broad part dominates in \eqref{broadnarrow2}.  For each $B \in \mathcal{B}$, by Lemma~\ref{broadplank}, the set $\ell^*(B) \cap B_E(0,2)$ is contained in a plank of dimensions $\sim 1 \times \delta \times \delta^2$, with long direction parallel to $u(\ell^*(B))$. Let $T_1(B), \dotsc, T_M(B)$ be $\sim 1 \times \delta^2 \times \delta^2$ tubes parallel to $u(\ell^*(B))$, which form a boundedly overlapping cover of this plank, so that $M \sim \delta^{-1}$ with $M$ independent of $B$. Then by \eqref{triplestar} and the trilinear Kakeya inequality (Theorem~\ref{kakeya}),
\begin{align*} &\left\lVert \sum_{B \in \mathcal{B}} a_B\chi_{\ell^*(B) } \right\rVert_{L^q(B_E(0,2)  \cap \ell^{-1}(\mathcal{L}_{\angle})) )}^q  \\
&\quad\lesssim \rho^{-4q/3}\int_{B_E(0,2)} \Bigg( \sum_{B_1,B_2, B_3 \in \mathcal{B}}  a_{B_1}a_{B_2}a_{B_3}\chi_{\ell^*(B_1)} \chi_{\ell^*(B_2)} \chi_{\ell^*(B_3)}\times \\
&\qquad \left\lvert u(\ell^*(B_1)) \wedge u(\ell^*(B_2)) \wedge u(\ell^*(B_3)) \right\rvert \Bigg)^{q/3} \\
&\quad\lesssim \rho^{-4q/3}\int_{\mathbb{R}^3} \Bigg( \sum_{B_1,B_2, B_3 \in \mathcal{B}} \sum_{1 \leq m_1, m_2, m_3 \leq M} a_{B_1}a_{B_2}a_{B_3} \times\\
&\qquad \chi_{T_{m_1}(B_1)}\chi_{T_{m_2}(B_2)}\chi_{{T_{m_3}(B_3)}}  \left\lvert u(\ell^*(B_1)) \wedge u(\ell^*(B_2)) \wedge u(\ell^*(B_3)) \right\rvert \Bigg)^{q/3} \\
&\quad \lesssim \rho^{-4q/3} \delta^{6} M^{3/2} \\
&\quad \lesssim  \rho^{-4q/3} \delta^{3q},  \end{align*}
since $q=3/2$. Since $c_t(\nu) \gtrsim \nu(\mathbb{H}) =1$, substituting this bound into \eqref{quadriplestar} proves \eqref{claimed2} when the broad part dominates, provided $C_t \gg \rho^{-4q/3}$ (where $\rho$ is yet to be chosen). This covers the case where the broad part dominates in \eqref{broadnarrow2}. \end{sloppypar}

If the narrow part dominates in \eqref{broadnarrow2}, then by \eqref{startwo},
\begin{multline*} \left\lVert \sum_{B \in \mathcal{B}} a_B\chi_{\ell^*(B) } \right\rVert_{L^q(B_E(0,2)  \cap \ell^{-1}(\mathcal{L}_{\angle}))}^q \\
\lesssim \sum_{\tau} \left\lVert \sum_{B \in \mathcal{B} : u(\ell^*(B)) \in \tau } a_B\chi_{\ell^*(B) } \right\rVert_{L^q(B_E(0,2) \cap \ell^{-1}(\mathcal{L}_{\angle}))}^q. \end{multline*}
For each $\tau$, let $\mathbb{T}_{\tau}$ be a boundedly overlapping cover of $\mathbb{R}^3$ by planks of dimensions $\sim \rho \times \rho^2 \times 1$ parallel to $\tau$. By Lemma~\ref{broadplank}, for each $B \in \mathcal{B}$ with $u(\ell^*(B)) \in \tau$, there exists $T \in \mathbb{T}_{\tau}$ such that $\ell^*(B) \cap B_E(0,2) \subseteq T$ (provided the implicit constant in the definition of the $T$'s is now chosen large enough). Hence
\begin{multline*} \sum_{\tau} \left\lVert \sum_{B \in \mathcal{B} : u(\ell^*(B)) \in \tau } a_B\chi_{\ell^*(B) } \right\rVert_{L^q(B_E(0,2) \cap \ell^{-1}(\mathcal{L}_{\angle}))}^q \\ \lesssim \sum_{\tau} \sum_{T \in \mathbb{T}_{\tau} } \left\lVert \sum_{\substack{B \in \mathcal{B} : u(\ell^*(B)) \in \tau\\
\ell^*(B) \cap B_E(0,2) \subseteq T}} a_B\chi_{\ell^*(B) } \right\rVert_{L^q(B_E(0,2) \cap \ell^{-1}(\mathcal{L}_{\angle}))}^q. \end{multline*}
The point-line duality step at Eq.~\eqref{reversible2} is reversible provided that the radii of the Korányi balls are doubled. More precisely, if $B \in \mathcal{B}$ is such that $B \cap \ell \neq \emptyset$ for some horizontal line $\ell$, then $\mathcal{H}^1_{\mathbb{H}}( 2B \cap \ell) \sim \delta$. This follows by left-translating to the origin and using $\mathcal{H}^1_{\mathbb{H}}\left( B_{\mathbb{H}}(0, \delta) \cap \mathbb{V}_{\theta}\right) \sim \delta$ for any $\theta \in [0, \pi)$; since $d_{\mathbb{H}}$ is equal to the Euclidean metric on any horizontal subgroup $\mathbb{V}_{\theta}$. This gives
\[ \int_{\pi/4}^{3\pi/4} \left\lVert P_{\mathbb{V}_{\theta}^{\perp}\#} \nu \right\rVert_{L^q\left(\mathcal{H}^3_{\mathbb{H}}\right)}^q \, d\theta \lesssim \sum_{\tau} \sum_{T \in \mathbb{T}_{\tau} }\int_{\pi/4}^{3\pi/4} \left\lVert P_{\mathbb{V}_{\theta}^{\perp}\#} \nu_T \right\rVert_{L^q\left(\mathcal{H}^3_{\mathbb{H}}\right)}^q \, d\theta, \]
where 
\[ \nu_T = \frac{1}{\delta^4}\sum_{\substack{B \in \mathcal{B} : u(\ell^*(B)) \in \tau\\
\ell^*(B) \cap B_E(0,2) \subseteq T}}  a_B\chi_{2B}. \]
By Lemma~\ref{localisation}, each measure $\nu_T$ is supported in a Korányi ball of radius $\sim \rho$. Let $(z_T,t_T)$ be the centre of this ball.  By Lemma~\ref{translation}, 
\[ \int_{\pi/4}^{3\pi/4} \left\lVert P_{\mathbb{V}_{\theta}^{\perp}\#} \nu_T \right\rVert_{L^q\left(\mathcal{H}^3_{\mathbb{H}}\right)}^q \, d\theta \leq \int_{0}^{\pi} \left\lVert P_{\mathbb{V}_{\theta}^{\perp}\#} L_{(z_T,t_T)^{-1}\#}\nu_T \right\rVert_{L^q\left(\mathcal{H}^3_{\mathbb{H}}\right)}^q \, d\theta. \]
For each $\lambda>0$, let $D_{\lambda} : \mathbb{H} \to \mathbb{H}$ be the dilation $(z,t) \mapsto (\lambda z, \lambda^2 t)$. Then
\begin{equation} \label{starsss} P_{\mathbb{V}_{\theta}^{\perp}} = D_{\rho}P_{\mathbb{V}_{\theta}^{\perp}}D_{\rho^{-1}}, \end{equation}
and therefore 
\[  P_{\mathbb{V}_{\theta}^{\perp}\#} L_{(z_T,t_T)^{-1}\#}\nu_T = D_{\rho\#}P_{\mathbb{V}_{\theta}^{\perp}\#}D_{\rho^{-1}\#}L_{(z_T,t_T)^{-1}\#}\nu_T. \]
This gives 
\begin{multline} \label{pause} \int_{0}^{\pi} \left\lVert P_{\mathbb{V}_{\theta}^{\perp}\#} L_{(z_T,t_T)^{-1}\#}\nu_T \right\rVert_{L^q\left(\mathcal{H}^3_{\mathbb{H}}\right)}^q \, d\theta \\
= \int_{0}^{\pi} \left\lVert D_{\rho\#} P_{\mathbb{V}_{\theta}^{\perp}\#}  D_{\rho^{-1}\#}L_{(z_T,t_T)^{-1}\#}\nu_T \right\rVert_{L^q\left(\mathcal{H}^3_{\mathbb{H}}\right)}^q \, d\theta. \end{multline}
For any $\theta \in [0, \pi)$, given a non-negative Borel function $f$ supported in $\mathbb{V}_{\theta}^{\perp}$ which is integrable with respect to $\mathcal{H}^3_{\mathbb{H}}$, if $\mu_f$ is the measure supported in $\mathbb{V}_{\theta}^{\perp}$ such that the Radon-Nikodym derivative of $\mu_f$ with respect to $\mathcal{H}^3_{\mathbb{H}}$ is equal to $f$, then the Radon-Nikodym derivative of $D_{\rho\#} \mu_f$ with respect to $\mathcal{H}^3_{\mathbb{H}}$ is equal to $\rho^{-3} \left(f \circ D_{\rho^{-1}} \right)$. By a change of variables, it follows that 
\begin{equation} \label{fivestars} \eqref{pause} = \rho^{-3(q-1)} \int_{0}^{\pi} \left\lVert P_{\mathbb{V}_{\theta}^{\perp}\#}  D_{\rho^{-1}\#}L_{(z_T,t_T)^{-1}\#}\nu_T \right\rVert_{L^q\left(\mathcal{H}^3_{\mathbb{H}}\right)}^q \, d\theta. \end{equation}
Moreover, 
\[ c_t\left( D_{\rho^{-1}\#}L_{(z_T,t_T)^{-1}\#}\nu_T\right) = \rho^t c_t(\nu_T) \lesssim \rho^t c_t(\nu). \]
Hence, by applying the induction hypothesis\footnote{By the triangle inequality and by the invariance of $L^q$ norms of projections under left translations of the measure, the induction hypothesis for measures supported in a Korányi ball of radius $\sim 1$ around the origin follows from the induction hypothesis for measures supported in $B_{\mathbb{H}}(0,c)$, with a worse constant.} at scale $\delta/\rho$, 
\begin{align} \notag &\sum_{\tau} \sum_{T \in \mathbb{T}_{\tau} }\int_{\pi/4}^{3\pi/4} \left\lVert P_{\mathbb{V}_{\theta}^{\perp}\#} \nu_T \right\rVert_{L^q\left(\mathcal{H}^3_{\mathbb{H}}\right)}^q \, d\theta \\
\label{hypothesis} &\quad \lesssim  \sum_{\tau} \sum_{T \in \mathbb{T}_{\tau} }  C_t\rho^{(t-3)(q-1)} c_t(\nu)^{q-1} \nu_T(\mathbb{H})\\
\label{trivial} &\quad \leq  C_t\rho^{(t-3)(q-1)} c_t(\nu)^{q-1}  \sum_{\tau} \sum_{T \in \mathbb{T}_{\tau} }\sum_{\substack{B \in \mathcal{B} : \\
 u(\ell^*(B)) \in \tau, \\
\ell^*(B) \cap B_E(0,2) \subseteq T }} \nu(B) \\
\label{rhopower} &\quad \lesssim C_t\rho^{(t-3)(q-1)} \nu(\mathbb{H}) c_t(\nu)^{q-1}. \end{align}
The inequality
\[ \nu_T(\mathbb{H}) \lesssim \sum_{\substack{B \in \mathcal{B} : \\
 u(\ell^*(B)) \in \tau, \\
\ell^*(B) \cap B_E(0,2) \subseteq T }} \nu(B),\]
used to get from \eqref{hypothesis} to \eqref{trivial}, follows from the definitions of $\nu_T$ and $\nu$. A justification will be given for the inequality
\begin{equation} \label{pause116} \sum_{\tau} \sum_{T \in \mathbb{T}_{\tau} }\sum_{\substack{B \in \mathcal{B} : \\
 u(\ell^*(B)) \in \tau, \\
\ell^*(B) \cap B_E(0,2) \subseteq T }} \nu(B) \lesssim \nu(\mathbb{H}), \end{equation}
used to get from \eqref{trivial} to \eqref{rhopower}. By Fubini, \eqref{pause116} is equivalent to
\begin{equation} \label{pause117} \sum_{B \in \mathcal{B} } \sum_{\substack{\tau: \\  u(\ell^*(B)) \in \tau}} \sum_{\substack{T \in \mathbb{T}_{\tau}:\\ \ell^*(B) \cap B_E(0,2) \subseteq T }} \nu(B) \lesssim \nu(\mathbb{H}). \end{equation}
For each $B \in \mathcal{B}$, there are $\lesssim 1$ pairs $(\tau, T)$ with $T \in \mathbb{T}_{\tau}$ such that $u(\ell^*(B)) \in \tau$ and $\ell^*(B) \cap B_E(0,2) \subseteq T$. Hence 
\[ \sum_{B \in \mathcal{B} } \sum_{\substack{\tau: \\  u(\ell^*(B)) \in \tau}} \sum_{\substack{T \in \mathbb{T}_{\tau}:\\ \ell^*(B) \cap B_E(0,2) \subseteq T }} \nu(B) \lesssim \sum_{B \in \mathcal{B} }  \nu(B). \]
The balls $B \in \mathcal{B}$ are disjoint from each other, so this verifies \eqref{pause117} and hence \eqref{pause116}.

The power of $\rho$ in \eqref{rhopower} is positive since $t>3$ and $q>1$, so the induction closes provided $\rho$ is sufficiently small (independent of $\delta$). This proves \eqref{claimed2}. 

As explained at the beginning of the proof, the inequality \eqref{claimed2} implies that \eqref{claimed7} holds for any $\delta \in (0,1)$. By Fatou's lemma, it follows that 
\[ \liminf_{\delta \to 0^+} \left\lVert P_{\mathbb{V}_{\theta}^{\perp}\#}\left(\mu \ast_{\mathbb{H}} \eta_{\delta}\right) \right\rVert_{L^{3/2}\left(\mathcal{H}^3_{\mathbb{H}}\right)} < \infty, \]
for a.e.~$\theta \in [0, \pi)$. By duality of $L^{3/2}$ and $L^3$, using that the compactly supported continuous functions are dense in $L^3$, and using that $\mu \ast_{\mathbb{H}} \eta_{\delta}$ converges to $\mu$ weak* as $\delta \to 0^+$, it follows that $P_{\mathbb{V}_{\theta}^{\perp}\#}\mu \in L^{3/2}\left(\mathcal{H}^3_{\mathbb{H}}\right)$ for a.e.~$\theta \in [0,\pi)$, and that 
\[ \left\lVert P_{\mathbb{V}_{\theta}^{\perp}\#}\mu \right\rVert_{L^{3/2}\left(\mathcal{H}^3_{\mathbb{H}}\right)} \leq \liminf_{\delta \to 0^+} \left\lVert P_{\mathbb{V}_{\theta}^{\perp}\#}\left(\mu \ast_{\mathbb{H}} \eta_{\delta}\right)\right\rVert_{L^{3/2}\left(\mathcal{H}^3_{\mathbb{H}}\right)}, \]
for a.e.~$\theta \in [0, \pi)$. By Fatou's lemma again,
\[ \int_0^{\pi} \left\lVert P_{\mathbb{V}_{\theta}^{\perp}\#}\mu \right\rVert^{3/2}_{L^{3/2}\left(\mathcal{H}^3_{\mathbb{H}}\right)} \, d\theta \lesssim \mu(\mathbb{H}) c_t(\mu)^{q-1}.  \]
This finishes the proof.
\end{proof}

\begin{theorem} \label{exceptionalset} If $s>3$ and $A \subseteq \mathbb{H}$ is $\mathcal{H}^s_{\mathbb{H}}$-measurable with $0 < \mathcal{H}^s_{\mathbb{H}}(A) < \infty$, then there is a set $E \subseteq [0, \pi)$ of measure zero, such that for any $\mathcal{H}^s_{\mathbb{H}}$-measurable set $B \subseteq A$ with $\mathcal{H}^s_{\mathbb{H}}(B)>0$, 
\[ \mathcal{H}^3_{\mathbb{H}}\left( P_{\mathbb{V}_{\theta}^{\perp}}(B) \right) >0 \quad \text{for all $\theta \in [0, \pi) \setminus E$.} \]
The set $E$ can be taken to be 
\[ E = \left\{ \theta \in [0, \pi): P_{\mathbb{V}_{\theta}^{\perp}\#} \mu \not\ll \mathcal{H}^3_{\mathbb{H}} \right\}, \]
where $\mu$ is defined by $\mu(F) = \mathcal{H}^s_{\mathbb{H}}(A \cap F)$ for any Borel set $F$.  \end{theorem} 
\begin{proof} It may be assumed that $s \leq 4$. An upper density inequality for general metric spaces (see e.g.~\cite[Exercise 2.4.4]{ambrosiotilli}) gives that for $\mathcal{H}^s_{\mathbb{H}}$-a.e.~$x \in A$, 
\begin{equation} \label{claim17} \limsup_{r \to 0^+} \frac{ \mathcal{H}^s_{\mathbb{H}}(A \cap B_{\mathbb{H}}(x,r) )}{r^s} < 100. \end{equation}
For each $j \geq 1$, define inductively
\[ A_j = \left\{ x \in A \cap B_{\mathbb{H}}(0, j) \setminus \bigcup_{0 < k < j} A_k : \sup_{0 < r < 1/j} \frac{ \mathcal{H}^s_{\mathbb{H}}(A \cap B_{\mathbb{H}}(x,r) )}{r^s}  <  100 \right\}. \]
Then each $A_j$ is an $\mathcal{H}^s_{\mathbb{H}}$-measurable subset of $A$, and $\mathcal{H}^s_{\mathbb{H}}\left(A \setminus \bigcup_{j=1}^{\infty} A_j\right) = 0$ by \eqref{claim17}. Hence if $\mu_j$ is is defined by $\mu_j(F) = \mu(F \cap A_j)$ for any Borel set $F$, then $c_s(\mu_j) < \infty$ for each $j$, and $\mu = \sum_j \mu_j$. By Theorem~\ref{quantitativetheorem},  
\[ P_{\mathbb{V}_{\theta}^{\perp}\#}\mu_j \ll \mathcal{H}^3_{\mathbb{H}}, \]
for a.e.~$\theta \in [0, \pi)$. Since $\mu = \sum_j \mu_j$, it follows that
\[ P_{\mathbb{V}_{\theta}^{\perp}\#}\mu \ll \mathcal{H}^3_{\mathbb{H}}, \] 
for a.e.~$\theta \in [0, \pi)$.

Now let $B \subseteq A$ be any $\mathcal{H}^s_{\mathbb{H}}$-measurable set with $\mathcal{H}^s_{\mathbb{H}}(B) >0$. Define $\nu$ by
\[ \nu(F) = \mu(F \cap B) \quad (=\mathcal{H}^s_{\mathbb{H}}(F \cap B)), \]
for any Borel set $F$. Then 
\[ \left\{ \theta \in [0, \pi): P_{\mathbb{V}_{\theta}^{\perp}\#} \nu \not\ll \mathcal{H}^3_{\mathbb{H}} \right\} \subseteq E, \]
so the theorem will follow from
\begin{equation} \label{inclusion} \left\{ \theta \in [0, \pi): \mathcal{H}^3_{\mathbb{H}}\left( P_{\mathbb{V}_{\theta}^{\perp}}(B) \right) =0 \right\}  \subseteq \left\{ \theta \in [0, \pi): P_{\mathbb{V}_{\theta}^{\perp}\#} \nu \not\ll \mathcal{H}^3_{\mathbb{H}} \right\}. \end{equation}
To see this, let $\theta \in [0, \pi)$ be such that $\mathcal{H}^3_{\mathbb{H}}\left( P_{\mathbb{V}_{\theta}^{\perp}}(B) \right) =0$. Let $F$ be a Borel set containing $P_{\mathbb{V}_{\theta}^{\perp}}(B)$ with $\mathcal{H}^3_{\mathbb{H}}(F) =0$. Then 
\[ 0 < \mathcal{H}^s_{\mathbb{H}}(B)  = \mu\left( \left(P_{\mathbb{V}_{\theta}^{\perp}}\right)^{-1}(F) \cap B \right) = \left(P_{\mathbb{V}_{\theta}^{\perp}\#} \nu\right)(F). \]
This shows that $P_{\mathbb{V}_{\theta}^{\perp}\#} \nu \not \ll \mathcal{H}^3_{\mathbb{H}}$ and verifies \eqref{inclusion}. \end{proof} 

\section{Intersections} \label{intersections}
The restriction of a measure $\mu$ to a measurable set $A$ is defined by $(\mu \restriction_A )(F) = \mu(F\cap A)$ for measurable sets $F$. The proof of the lemma below does not differ substantially from the proof of Lemma~3.2 from \cite{mattila2} (see also \cite[Lemma~16]{marstrand}), but the proof is included for completeness.
\begin{lemma} \label{analysis} Let $E \subseteq \mathbb{H}$ be a Borel set, $t>0$ and $\theta \in [0, \pi)$. If $\mathcal{H}^t_{\mathbb{H}}\left( E \cap P_{\mathbb{V}_{\theta}^{\perp}}^{-1}(u) \right) = 0$ for all $u \in \mathbb{V}_{\theta}^{\perp}$, then for any finite Borel measure $\mu$ on $\mathbb{H}$, 
\[ \limsup_{r \to 0^+} \liminf_{\delta \to 0^+} r^{-t} \delta^{-3} \mu\left\{ y \in B_{\mathbb{H}}(x,r) : d_{\mathbb{H}}\left(P_{\mathbb{V}_{\theta}^{\perp}}(x),P_{\mathbb{V}_{\theta}^{\perp}}(y) \right) < \delta \right\} = +\infty, \]
for $\mu$-a.e.~$x \in E$.   \end{lemma} 
\begin{proof} 
Since finite Borel measures are inner regular, and since
\begin{align*} &E \setminus \bigg\{ x \in E : \limsup_{r \to 0^+} \liminf_{\delta \to 0^+} \\
&\qquad r^{-t} \delta^{-3} \mu\left\{ y \in B_{\mathbb{H}}(x,r) : d_{\mathbb{H}}\left(P_{\mathbb{V}_{\theta}^{\perp}}(x),P_{\mathbb{V}_{\theta}^{\perp}}(y) \right) < \delta \right\} = +\infty \bigg\} \\
&\quad = \bigcup_{N =1}^{\infty}
\bigg\{ x \in E : \sup_{0 < r < 1/N} \liminf_{\delta \to 0^+} \\
&\qquad r^{-t} \delta^{-3} \mu\left\{ y \in B_{\mathbb{H}}(x,r) : d_{\mathbb{H}}\left(P_{\mathbb{V}_{\theta}^{\perp}}(x),P_{\mathbb{V}_{\theta}^{\perp}}(y) \right) < \delta \right\} \leq N \bigg\}, \end{align*}
it suffices to prove that, for any $N \geq 1$, $\mu(F) = 0$ for any nonempty compact set $F$ with
\begin{multline*} F \subseteq \Bigg\{ x \in E : \sup_{0< r < 1/N} \liminf_{\delta \to 0^+} \\
r^{-t} \delta^{-3} \mu\left\{ y \in B_{\mathbb{H}}(x,r) : d_{\mathbb{H}}\left(P_{\mathbb{V}_{\theta}^{\perp}}(x),P_{\mathbb{V}_{\theta}^{\perp}}(y) \right) < \delta \Bigg\} \leq N \right\}. \end{multline*}

 Fix such a set $F$. By partitioning $F$ into sets of smaller diameter, it may be assumed that $F$ is contained in a Korányi ball of radius $1/(100N)$. It will first be shown that $P_{\mathbb{V}_{\theta}^{\perp}\#}\left( \mu \restriction_F \right) \ll \mathcal{H}^3_{\mathbb{H}}$. Given 
\[ u \in \supp  P_{\mathbb{V}_{\theta \#}^{\perp}} \left( \mu \restriction_F \right) = P_{\mathbb{V}_{\theta}^{\perp}}\left( \supp\left( \mu \restriction_F \right) \right), \]
let $x_0 \in F$ be such that $P_{\mathbb{V}_{\theta}^{\perp}}(x_0) = u$. Then by the definition of $F$, 
\begin{align*} &\liminf_{\delta \to 0^+}  \frac{P_{\mathbb{V}_{\theta \#}^{\perp}}\left( \mu\restriction_{F} \right)\left(B_{\mathbb{H}}(u, \delta) \right)}{\delta^3} \\
&\quad = \liminf_{\delta \to 0^+}  \frac{P_{\mathbb{V}_{\theta \#}^{\perp}}\left( \mu\restriction_{F} \right)\left(B_{\mathbb{H}}\left(P_{\mathbb{V}_{\theta}^{\perp}}(x_0), \delta\right) \right)}{\delta^3} \\
&\quad\leq \liminf_{\delta \to 0^+}  \frac{\mu\left\{ x \in B_{\mathbb{H}}(x_0,1/(2N)) : d_{\mathbb{H}}\left(P_{\mathbb{V}_{\theta}^{\perp}}(x),P_{\mathbb{V}_{\theta}^{\perp}}(x_0) \right) < \delta \right\}}{\delta^3} \\
&\quad \lesssim N. \end{align*}
By a straightforward argument using the Vitali covering lemma and that the $\mathcal{H}_{\mathbb{H}}^3$-measure of any Korányi $\delta$-ball in $\mathbb{V}_{\theta}^{\perp}$ is $\sim \delta^3$, it follows that 
\[ P_{\mathbb{V}_{\theta \#}^{\perp}} \left( \mu \restriction_{F} \right) \ll \mathcal{H}_{\mathbb{H}}^3. \]

\begin{sloppypar} Hence, for any $0 < r < 1/(2N)$ and $x_0 \in F$, by the (generalised) Lebesgue differentiation theorem \cite[p.~13]{stein}, the Radon-Nikodym derivative of $P_{\mathbb{V}_{\theta \#}^{\perp}} \left( \mu \restriction_{F \cap B_{\mathbb{H}}(x_0,r)} \right)$ with respect to $\mathcal{H}_{\mathbb{H}}^3$ satisfies
\[ P_{\mathbb{V}_{\theta \#}^{\perp}} \left( \mu \restriction_{F \cap B_{\mathbb{H}}(x_0,r)} \right)(u) = c\lim_{ \delta \to 0^+ } \frac{P_{\mathbb{V}_{\theta \#}^{\perp}}\left( \mu\restriction_{F \cap B_{\mathbb{H}}(x_0,r)} \right)\left(B_{\mathbb{H}}(u, \delta) \right)}{\delta^3}, \]
for $\mathcal{H}^3_{\mathbb{H}}$-a.e.~$u \in \mathbb{V}_{\theta}^{\perp}$, where $c>0$ is an absolute constant. If $u \notin P_{\mathbb{V}_{\theta}^{\perp}}\left(F \cap \overline{B_{\mathbb{H}}(x_0,r)}\right)$, then the above limit is zero. Otherwise, there exists $x' \in F \cap \overline{B_{\mathbb{H}}(x_0,r)}$ such that $P_{\mathbb{V}_{\theta}^{\perp}}(x') = u$. Hence, by the definition of $F$, using that $0 < 2r < 1/N$ and $x' \in F$,
\begin{multline*} \lim_{ \delta \to 0^+ } \frac{P_{\mathbb{V}_{\theta \#}^{\perp}}\left( \mu\restriction_{F \cap B_{\mathbb{H}}(x_0,r)} \right)\left(B_{\mathbb{H}}(u, \delta) \right)}{\delta^3} \\
\leq \liminf_{ \delta \to 0^+ } \frac{\mu\left\{ x \in B_{\mathbb{H}}(x',2r) : d_{\mathbb{H}}\left(P_{\mathbb{V}_{\theta}^{\perp}}(x),P_{\mathbb{V}_{\theta}^{\perp}}(x') \right) < \delta \right\}}{\delta^3} \lesssim N r^t, \end{multline*}
for $\mathcal{H}^3_{\mathbb{H}}$-a.e.~$u \in \mathbb{V}_{\theta}^{\perp}$. Therefore, for any $\delta>0$, $x_0 \in F$ and $0 < r < 1/(2N)$,
\begin{multline} \label{fractal} P_{\mathbb{V}_{\theta \#}^{\perp}} \left( \mu \restriction_{F \cap B_{\mathbb{H}}(x_0,r) } \right)\left(B_{\mathbb{H}}\left(P_{\mathbb{V}_{\theta}^{\perp}}(x_0), \delta\right) \right) \\
 = \int_{B_{\mathbb{H}}\left(P_{\mathbb{V}_{\theta}^{\perp}}(x_0), \delta\right)} P_{\mathbb{V}_{\theta \#}^{\perp}} \left( \mu \restriction_{F \cap B_{\mathbb{H}}(x_0,r)} \right)(u) \, d\mathcal{H}^3_{\mathbb{H}}(u) \lesssim N r^t\delta^3. \end{multline}
This establishes the key inequality to be used below.\end{sloppypar}

Since 
\[ \mu(F) = P_{\mathbb{V}_{\theta \#}^{\perp}} \left( \mu \restriction_F \right)\left(\mathbb{V}_{\theta}^{\perp}\right) = \int_{\mathbb{V}_{\theta}^{\perp}} P_{\mathbb{V}_{\theta \#}^{\perp}} \left( \mu \restriction_F \right)(u) \, d\mathcal{H}^3_{\mathbb{H}}(u), \]
there must exist $u \in P_{\mathbb{V}_{\theta}^{\perp}}\left(F \cap \supp \mu\right)$ such that 
\[ P_{\mathbb{V}_{\theta \#}^{\perp}} \left( \mu \restriction_F \right)(u) \sim \lim_{ \delta \to 0^+ } \frac{P_{\mathbb{V}_{\theta \#}^{\perp}}\left( \mu\restriction_{F} \right)\left(B_{\mathbb{H}}(u, \delta) \right)}{\delta^3} \gtrsim C_F^{-1} \mu(F), \]
 where $C_F := \mathcal{H}^3_{\mathbb{H}}\left(P_{\mathbb{V}_{\theta}^{\perp}}(F)\right)$ (which is finite, and may be assumed nonzero since this would already imply that $\mu(F) = 0$). Hence there exists a $\delta_u>0$ such that for all $0 < \delta < \delta_u$, 
\[ \mu\left( F \cap P_{\mathbb{V}_{\theta}^{\perp}}^{-1}\left(B_{\mathbb{H}}(u,\delta) \right) \right) \gtrsim \delta^3 C_F^{-1} \mu(F). \]
But $\mathcal{H}^t_{\mathbb{H}}\left( F \cap P _{\mathbb{V}_{\theta}^{\perp}}^{-1}(u) \right) =0$ by assumption, and $F \cap P_{\mathbb{V}_{\theta}^{\perp}}^{-1}(u)$ is compact, so for any $\epsilon >0$ there exists a finite covering $\{B_{\mathbb{H}}(y_k, r_k)\}_k$ of $F \cap P_{\mathbb{V}_{\theta}^{\perp}}^{-1}(u)$ by Korányi balls, with centres $y_k \in F \cap P_{\mathbb{V}_{\theta}^{\perp}}^{-1}(u)$, such that 
\[ \sum_k r_k^t < \epsilon. \]
Hence, for sufficiently small $\epsilon >0$, and sufficiently small $\delta>0$,
\begin{align*} \mu(F) &\lesssim  C_F\delta^{-3}  \mu\left( F \cap P_{\mathbb{V}_{\theta}^{\perp}}^{-1}\left(B_{\mathbb{H}}(u,\delta) \right) \right) \\
&\leq C_F\sum_k \delta^{-3}\mu\left( F \cap B_{\mathbb{H}}(y_k, r_k) \cap P_{\mathbb{V}_{\theta}^{\perp}}^{-1}\left(B_{\mathbb{H}}\left(P_{\mathbb{V}_{\theta}^{\perp} }(y_k),\delta\right) \right) \right) \\
&\lesssim C_F N\sum_k r_k^t &&\text{(by \eqref{fractal})}\\
&< C_F N\epsilon. \end{align*}
Letting $\epsilon \to 0$ gives $\mu(F) = 0$, which proves the lemma. \end{proof} 
 
\begin{theorem} Let $s>3$. If $A \subseteq \mathbb{H}$ is $\mathcal{H}^s_{\mathbb{H}}$-measurable with $0 < \mathcal{H}^s_{\mathbb{H}}(A) < \infty$, then for $\left(\mathcal{H}^s_{\mathbb{H}} \times \mathcal{H}^1_E\right)$-a.e.~$(x, \theta) \in A \times [0, \pi)$, 
\begin{equation} \label{slicingeq} \dim\left( A \cap P_{\mathbb{V}_{\theta}^{\perp}}^{-1}\left(P_{\mathbb{V}_{\theta}^{\perp}}(x) \right)\right) = s-3, \end{equation}
and for a.e.~$\theta \in [0,\pi)$,
\begin{equation} \label{measureintersect} \mathcal{H}^3_{\mathbb{H}}\left\{ w \in \mathbb{V}_{\theta}^{\perp} : \dim\left( A \cap P_{\mathbb{V}_{\theta}^{\perp}}^{-1}(w) \right) = s-3 \right\} >0. \end{equation} \end{theorem} \begin{sloppypar}
\begin{proof} For every $\theta \in [0, \pi)$, the inequality 
\begin{equation} \label{upperbound} \dim\left( A \cap P_{\mathbb{V}_{\theta}^{\perp}}^{-1}(w) \right) \leq s-3, \end{equation}
holds for $\mathcal{H}^3_{\mathbb{H}}$-a.e.~$w \in \mathbb{V}_{\theta}^{\perp}$, as a consequence of the following inequality\footnote{This inequality shows that both conclusions of the theorem can be refined to say that the intersections have dimension $s-3$ and finite $\mathcal{H}^{s-3}_{\mathbb{H}}$-measure. This refinement is a Heisenberg group version of Theorem~III from \cite{marstrand}.} from \cite[Lemma~5.3]{chousionisfasslerorponen}:
\[ \int^{*} \mathcal{H}^{s-3}_{\mathbb{H}}\left( A \cap P_{\mathbb{V}_{\theta}^{\perp}}^{-1}(w) \right) \, d\mathcal{H}^3_{\mathbb{H}}(w) \lesssim \mathcal{H}^s_{\mathbb{H}}(A). \]
Recall that the upper integral $\int^* f \, d\mu$ of $f: X \to [0, +\infty]$ on a measure space $(X, \mu)$ is the infimum of $\int \phi \, d\mu$ over the measurable functions $\phi : X \to [0, +\infty]$ with $f \leq \phi$ (see e.g.~\cite[p.~13]{mattila}). The finiteness of $\int^{*} \mathcal{H}^{s-3}_{\mathbb{H}}\left( A \cap P_{\mathbb{V}_{\theta}^{\perp}}^{-1}(w) \right) \, d\mathcal{H}^3_{\mathbb{H}}(w)$ implies that
\[ \mathcal{H}^{s-3}_{\mathbb{H}}\left( A \cap P_{\mathbb{V}_{\theta}^{\perp}}^{-1}(w) \right) < +\infty, \]
 for $\mathcal{H}^3_{\mathbb{H}}$-a.e.~$w \in \mathbb{V}_{\theta}^{\perp}$, and this implies \eqref{upperbound} for $\mathcal{H}^3_{\mathbb{H}}$-a.e.~$w \in \mathbb{V}_{\theta}^{\perp}$. This yields the upper bound in \eqref{measureintersect}, and by Theorem~\ref{exceptionalset} it also shows that for $\left(\mathcal{H}^s_{\mathbb{H}} \times \mathcal{H}^1_E\right)$-a.e.~$(x, \theta) \in A \times [0, \pi)$, 
\[ \dim\left( A \cap P_{\mathbb{V}_{\theta}^{\perp}}^{-1}\left(P_{\mathbb{V}_{\theta}^{\perp}}(x) \right)\right) \leq s-3. \]

It remains to prove the lower bounds. Using~\eqref{claim17} again, by letting $A_j$ be bounded $\mathcal{H}^s_{\mathbb{H}}$-measurable sets such that $c_s\left( \mathcal{H}^s_{\mathbb{H}}\restriction_{A_j}\right) < \infty$ for all $j$ and $\mathcal{H}^s_{\mathbb{H}}\left( A \setminus \bigcup_j A_j \right) = 0$, it may be assumed that the restriction $\mu$ of $\mathcal{H}^s_{\mathbb{H}}$ to $A$ satisfies $c_s(\mu) < \infty$, and by dilation it may be assumed that $A \subseteq B_{\mathbb{H}}(0,1)$. Using Theorem~1.9, Corollary~1.11 and Theorem~4.2 from \cite{mattila}, by replacing $A$ with a compact $A' \subseteq A$ such that $\mathcal{H}^s_{\mathbb{H}}(A \setminus A')$ is small, it may be assumed that $A$ is compact. Let $t \in (0,s-3)$. Let $\delta>0$ and let $r> \delta$. Let $\{B_j\}_j$ be a boundedly overlapping cover of $B_{\mathbb{H}}(0,1)$ by Korányi balls of radius $r$. Let $\mu_j$ be the restriction of $\mu$ to $2B_j$. Then 
\begin{multline*} \int_0^{\pi} \int \left( r^{-t} \delta^{-3} \mu\left\{ y \in B_{\mathbb{H}}(x,r) : d_{\mathbb{H}}\left( P_{\mathbb{V}_{\theta}^{\perp} }(y), P_{\mathbb{V}_{\theta}^{\perp} }(x) \right) < \delta \right\} \right)^{1/2} \, d\mu(x) \, d\theta  \\
\lesssim \sum_j \int_0^{\pi} \int  \left( r^{-t} \delta^{-3} \mu_j\left\{ y \in \mathbb{H} : d_{\mathbb{H}}\left( P_{\mathbb{V}_{\theta}^{\perp} }(y), P_{\mathbb{V}_{\theta}^{\perp} }(x) \right) < \delta \right\} \right)^{1/2} \, d\mu_j(x) \, d\theta.  \end{multline*}
For each $j$ and each $\theta \in [0, \pi)$,
\begin{multline*} \int \left( r^{-t} \delta^{-3} \mu_j\left\{ y \in \mathbb{H} : d_{\mathbb{H}}\left( P_{\mathbb{V}_{\theta}^{\perp} }(y), P_{\mathbb{V}_{\theta}^{\perp} }(x) \right) < \delta \right\}\right)^{1/2} \, d\mu_j(x)  \\
=  \int \left( r^{-t} \delta^{-3}\left( P_{\mathbb{V}_{\theta\#}^{\perp}}\mu_j\right)\left(B_{\mathbb{H}}(x, \delta)\right)\right)^{1/2} \, d\left(P_{\mathbb{V}_{\theta\#}^{\perp}}\mu_j \right)(x). \end{multline*}
Let $\eta = \chi_{B_{\mathbb{H}}(0,1) \cap \mathbb{V}_{\theta}^{\perp}}$, and let $\eta_{\delta}(z,t) =\delta^{-3} \eta(z/\delta, t/\delta^2)$. Then the above can be written as 
\[  r^{-t/2} \int \left[\left( \eta_{\delta} \ast_{\theta} P_{\mathbb{V}_{\theta\#}^{\perp}}\mu_j\right)(x)\right]^{1/2} \, d\left(P_{\mathbb{V}_{\theta\#}^{\perp}}\mu_j \right)(x), \]
where $\ast_{\theta}$ refers to convolution in $\mathbb{V}_{\theta}^{\perp}$ (which is commutative). By Hölder's inequality with respect to the measure $\mathcal{H}^3_{\mathbb{H}}$ on $\mathbb{V}_{\theta}^{\perp}$, this is 
\[ \leq r^{-t/2} \left\lVert \eta_{\delta} \ast_{\theta} P_{\mathbb{V}_{\theta\#}^{\perp}}\mu_j \right\rVert_{L^{3/2}\left(\mathcal{H}^3_{\mathbb{H}}\right) }^{1/2} \left\lVert P_{\mathbb{V}_{\theta\#}^{\perp}}\mu_j \right\rVert_{L^{3/2}\left(\mathcal{H}^3_{\mathbb{H}}\right)}. \]
By Young's convolution inequality, using $\left\lVert \eta_{\delta}\right\rVert_{L^{1}\left(\mathcal{H}^3_{\mathbb{H}}\right) } \sim 1$, this is
\[ \lesssim r^{-t/2} \left\lVert P_{\mathbb{V}_{\theta\#}^{\perp}}\mu_j \right\rVert_{L^{3/2}\left(\mathcal{H}^3_{\mathbb{H}}\right)}^{3/2}. \] 
Integrating over $\theta$ and summing over $j$ gives
\begin{multline} \label{pause37} \int_0^{\pi} \int \left( r^{-t} \delta^{-3} \mu\left\{ y \in B_{\mathbb{H}}(x,r) : d_{\mathbb{H}}\left( P_{\mathbb{V}_{\theta}^{\perp} }(y), P_{\mathbb{V}_{\theta}^{\perp} }(x) \right) < \delta \right\} \right)^{1/2} \, d\mu(x) \, d\theta \\
\lesssim r^{-t/2} \sum_j \int_0^{\pi} \left\lVert P_{\mathbb{V}_{\theta\#}^{\perp}} \mu_j \right\rVert_{L^{3/2}\left(\mathcal{H}^3_{\mathbb{H}}\right)}^{3/2} \, d\theta. \end{multline}
Let $x_j$ be the centre of $B_j$. Then by Lemma~\ref{translation}, 
\[ \eqref{pause37} = r^{-t/2} \sum_j \int_0^{\pi} \left\lVert P_{\mathbb{V}_{\theta\#}^{\perp}} L_{-x_j \#}\mu_j \right\rVert_{L^{3/2}\left(\mathcal{H}^3_{\mathbb{H}}\right)}^{3/2} \, d\theta. \]
By rescaling as in~\eqref{fivestars}, this equals
\[ r^{(-3-t)/2} \sum_j \int_0^{\pi} \left\lVert P_{\mathbb{V}_{\theta\#}^{\perp}} D_{r^{-1}\#}L_{-x_j \#}\mu_j \right\rVert_{L^{3/2}\left(\mathcal{H}^3_{\mathbb{H}}\right)}^{3/2} \, d\theta. \]
Using the inequality 
\[ c_s( D_{r^{-1}\#} L_{-x_j \#}\mu_j) \leq c_s( D_{r^{-1}\#} L_{-x_j \#}\mu)  = r^s c_s(\mu), \]
and then Theorem~\ref{quantitativetheorem}, gives
\[ \eqref{pause37} \lesssim r^{(s-3-t)/2} c_s(\mu)^{1/2} \sum_j \mu(2B_j) \lesssim  r^{(s-3-t)/2} c_s(\mu)^{1/2} \mu(\mathbb{H}). \]
Let $k_0$ be a large integer. Since $t < s-3$, summing over $r = 2^{-k}$, $k \geq k_0$, yields 
\begin{multline*} \sum_{k \geq k_0} \liminf_{\delta \to 0^+}\int_0^{\pi} \int \left( 2^{-kt} \delta^{-3} \mu\left\{ y \in B_{\mathbb{H}}\left(x,2^{-k}\right) : d_{\mathbb{H}}\left( P_{\mathbb{V}_{\theta}^{\perp} }(y), P_{\mathbb{V}_{\theta}^{\perp} }(x) \right) < \delta \right\} \right)^{1/2} \\ d\mu(x) \, d\theta 
 \lesssim 2^{-k_0(s-3-t)/2} c_s(\mu)^{1/2} \mu(\mathbb{H}). \end{multline*}
By the monotone convergence theorem and Fatou's lemma, this gives 
\begin{multline*} \int_0^{\pi} \int \limsup_{r \to 0^+} \liminf_{\delta \to 0^+} \\
\left( r^{-t} \delta^{-3} \mu\left\{ y \in B_{\mathbb{H}}(x,r) : d_{\mathbb{H}}\left( P_{\mathbb{V}_{\theta}^{\perp} }(y), P_{\mathbb{V}_{\theta}^{\perp} }(x) \right) < \delta \right\} \right)^{1/2} \, d\mu(x) \, d\theta =0. \end{multline*}
It follows that 
\begin{equation} \label{zeros} \lim_{r \to 0^+} \liminf_{\delta \to 0^+}  r^{-t} \delta^{-3} \mu\left\{ y \in B_{\mathbb{H}}(x,r) : d_{\mathbb{H}}\left( P_{\mathbb{V}_{\theta}^{\perp} }(y), P_{\mathbb{V}_{\theta}^{\perp} }(x) \right) < \delta \right\}  =0, \end{equation}
for $\left(\mu \times \mathcal{H}^1_E\right)$-a.e.~$(x, \theta) \in A \times [0, \pi)$.  For each $\theta \in [0, \pi)$, let 
\[ E_{\theta} = \left\{ x \in A : \mathcal{H}_{\mathbb{H}}^t\left( A \cap P_{\mathbb{V}_{\theta}^{\perp}}^{-1}\left(P_{\mathbb{V}_{\theta}^{\perp}}(x) \right) \right) = 0 \right\}. \]
Then each $E_{\theta}$ is a Borel set by Tonelli's theorem (see e.g.~\cite[Theorem~1.7.15]{tao}), and $\mathcal{H}^t_{\mathbb{H}}\left( E_{\theta} \cap P_{\mathbb{V}_{\theta}^{\perp}}^{-1}(u) \right) = 0$ for all $u \in \mathbb{V}_{\theta}^{\perp}$. Hence, by Lemma~\ref{analysis}, 
\begin{equation} \label{infinity} \limsup_{r \to 0^+} \liminf_{\delta \to 0^+} r^{-t} \delta^{-3} \mu\left\{ y \in B_{\mathbb{H}}(x,r) : d_{\mathbb{H}}\left(P_{\mathbb{V}_{\theta}^{\perp}}(x),P_{\mathbb{V}_{\theta}^{\perp}}(y) \right) < \delta \right\} = \infty, \end{equation}
for $\mu$-a.e.~$x \in E_{\theta}$. By comparing \eqref{zeros} with \eqref{infinity} and applying Fubini's theorem, it follows that $\mu(E_{\theta}) = 0$ for a.e.~$ \theta \in [0, \pi)$. By Fubini's theorem again, this gives 
\[ \left(\mu \times \mathcal{H}^1_E\right)\left\{ (x, \theta) \in A \times [0, \pi) : x \in E_{\theta} \right\} = 0, \]
and thus $\dim\left( A \cap P_{\mathbb{V}_{\theta}^{\perp}}^{-1}\left(P_{\mathbb{V}_{\theta}^{\perp}}(x) \right) \right) \geq t$ for $\left(\mu \times \mathcal{H}^1_E\right)$-a.e.~$(x,\theta) \in A \times [0, \pi)$. Letting $t \to s-3$ from below gives $\dim\left( A \cap P_{\mathbb{V}_{\theta}^{\perp}}^{-1}\left(P_{\mathbb{V}_{\theta}^{\perp}}(x) \right) \right) \geq s-3$ for $\left(\mu \times \mathcal{H}^1_E\right)$-a.e.~$(x,\theta) \in A \times [0, \pi)$. The inequality~\eqref{measureintersect} then follows from \eqref{slicingeq} and Theorem~\ref{exceptionalset}.    \end{proof} \end{sloppypar}

\end{document}